\documentclass[12pt,a4paper]{amsart} 
\usepackage[a4paper,margin=1in,footskip=0.25in]{geometry}

\usepackage{amsmath, amssymb, amsthm, amsfonts}
\usepackage{hyperref}
\usepackage[capitalise,noabbrev]{cleveref}
\usepackage[all]{xy}
\newtheorem{introtheorem}{Theorem}

\newtheorem{thm}{Theorem}[section]

\newtheorem{lem}[thm]{Lemma}
\newtheorem{prop}[thm]{Proposition}
\theoremstyle{definition}
\newtheorem{defn}[thm]{Definition}
\theoremstyle{remark}
\newtheorem{rmk}[thm]{Remark}
\newtheorem{exm}[thm]{Example}

\newcommand{\CC}{\mathcal{C}}
\newcommand{\CD}{\mathcal{D}}
\newcommand{\CP}{\mathcal{P}}
\def\NN{\ensuremath{\mathbb{N}}}
\def\QQ{\ensuremath{\mathbb{Q}}}
\def\ZZ{\ensuremath{\mathbb{Z}}}

\renewcommand{\ker}{\operatorname{Ker}}
\renewcommand{\epsilon}{\varepsilon}
\newcommand{\im}{\operatorname{Im}}
\newcommand{\supp}{\operatorname{supp}}
\newcommand{\Fun}{\operatorname{Fun}}

\newcommand{\colim}{\operatorname{colim}}
\newcommand{\hocolim}{\operatorname{hocolim}}
\renewcommand{\lim}{\operatorname{lim}}

\newcommand{\DCC}{\operatorname{DCC}}
\newcommand{\Ho}{\operatorname{Ho}}
\newcommand{\Aut}{\operatorname{Aut}}
\newcommand{\End}{\operatorname{End}}
\newcommand{\CAT}{\operatorname{CAT}}
\newcommand{\PSL}{\operatorname{PSL}}

\newcommand{\Ab}{{\operatorname{Ab}}}
\newcommand{\Top}{{\operatorname{Top}}}
\newcommand{\SSets}{{\operatorname{SSets}}}
\newcommand{\Ch}{\operatorname{Ch}}
\newcommand{\lmod}{{\operatorname{-mod}}}
\newcommand{\op}{{\operatorname{op}}}
\newcommand{\direct}[1]{\overrightarrow{#1}}
\newcommand{\invers}[1]{\overleftarrow{#1}}

\usepackage{soul}
\usepackage{todonotes}
\title{Mackey functors for posets}
\author[G. Carrión Santiago]{Guille Carrión Santiago}
\email{g.carrionsantiago@uma.es}
\author[A. Díaz ramos]{Antonio Díaz Ramos}
\email{adiazramos@uma.es}
\address{Departamento de Ágebra, Geometría y Topología, Universidad de Málaga, Málaga, Spain}
\date{2023}

\date{\today}
\begin{document}
\maketitle
\begin{abstract}
We characterize cofibrant objects in the category of functors indexed in a filtered poset and we show that these objects are acyclic. As a consequence, we show that Mackey functors over posets are also acyclic, where we define this type of Mackey functors mimicking the classical notion. As application, we study homotopy colimits over posets and we give a homology decomposition for the classifying space of the Bianchi group $\Gamma_1$.
\end{abstract}

\section{Introduction}
Functors with values in abelian groups, as well as their higher direct and inverse limits, often occur in homotopy theory and group theory. For instance, in (co)homological decompositions \cite{D1997}, obstruction theory for maps \cite{JMO1994}, existence and uniqueness of the classifying space for fusion systems \cite{BLO2003} or Alperin weight conjecture \cite{Linkclemann2005}. Fundamental vanishing results for higher limits, i.e., for $\lim$-acyclicity or $\colim$-acyclicity of a functor, are related to Mackey functors \cite{JM1992, DP2015} and $\Lambda$-functors \cite{JMOI1992,JMOII1992}. Under mild assumptions, higher limits over a category may be reduced to higher limits over a poset \cite{S1999} and, for posets, well-known conditions for acyclicity are the Mittag-Leffler condition and directed posets.

Within the context of homological algebra, higher limits are the derived functors of direct and inverse limits. In this work, we follow the point of view of homotopy to study these derived functors over posets. More precisely, if $\CP$ is a filtered poset, i.e., a poset equipped with an order-preserving map $\CP\to \NN$, and $R$ is a commutative ring with unit, we endow the category of functors $\Fun(\CP,\Ch(R))$ with a structure of model category such that a cofibrant functor ${F\colon \CP\to R\lmod}$ is $\colim$-acyclic, where we see $F$ concentrated in degree $0$ in the category of unbounded chain complexes $\Ch(R)$. It turns out that $F$ is cofibrant if and only if
\begin{equation}\label{equ:cofibrant_to_Ch(Ab)}
\colim_{\CP_{<i}}F\to F(i)\text{ is a monomorphism for every $i\in \CP$.}
\end{equation}
We show that this condition is equivalent, over a poset $\CP$ satisfying the descending chain condition ($\DCC$ poset for short), to the notion of pseudo-projective functor.

\begin{defn}\label{def:pseudo-projective}
A functor $F\colon \CP\to R\lmod$ over a $\DCC$ poset is \emph{pseudo-projective} at $i\in \CP$ if, for every finite subset $J\subset \CP_{\le i}$ and every element $\oplus_{j\in J}x_j\in\bigoplus_{j\in J} F(j)$, the condition:
\[
\sum_{j\in J} F(j< i)(x_j)=0
\]
implies that $x_j\in \im_F(j)=\sum_{k<j}\im F(k< j)$ for every $j\in \max J$. We say that $F$ is \emph{pseudo-projective} if $F$ is pseudo-projective at $i\in \CP$ for every $i\in \CP$. 
\end{defn}

Pseudo-projectivity first appeared in \cite{DiazRamos2009}, where it is shown that a, over a $\DCC$ poset $\CP$, a functor ${F\colon \CP\to R\lmod}$ is projective if and only if it is pseudo-projective and $F(i)/\im_F(i)$ is free for any $i\in P$.

\begin{introtheorem}\label{thm:pseudo-sii-locally_over_DCC}
Let $\CP$ be a $\DCC$ poset, and $F\colon \CP\to R\lmod$ be a functor. Then, $F$ satisfies \eqref{equ:cofibrant_to_Ch(Ab)} if and only if 
$F$ is pseudo-projective.
\end{introtheorem}

As filtered posets satisfy $\DCC$, we also obtain the following.

\begin{introtheorem}\label{thm:pseudo-sii-locally}
Let $\CP$ be a filtered poset and $F\colon \CP\to R\lmod$ be a functor. Then $F$ is cofibrant if and only if $F$ is pseudo-projective. Moreover, in this case, $F$ is $\colim$-acyclic.
\end{introtheorem}

The part of $\colim$-acyclicity easily follows from model category theory, hence providing an alternative proof of \cite[Theorem B]{DiazRamos2009} when $\CP$ is filtered. In addition, this theorem is useful in applications as it is often easier to study pseudo-projectivity than cofibrancy.
As a first application of pseudo-projectivity, and inspired by the following example, we provide a definition of Mackey functors for posets below. 

\begin{exm}
\label{exm:ClassicMackeyNormal}
Given a commutative ring with unit $R$, a \emph{Mackey functor} for a finite group $G$ \cite{DRESS1973},  consists of a pair of functors $(M_*, M^*)$ from the category of finite $G$-sets to the category of $R$-modules, such that $M_*$ is covariant, $M^*$ is contravariant, both coincide on objects and take disjoint union to direct sum, and they satisfy the Mackey formula:
\begin{equation}
\label{Mackey/Mackey equation}
M^*([\iota_V^U])\circ M_*[\iota_W^U])=\sum_{x\in [V\backslash U/W]} M_*([\iota_{V\cap^{x}\!W}^ V])\circ M_*([c_x])\circ M^*([\iota_{V^x\cap W}^ W])
\end{equation}
where $V,W\le U$ , $\iota_\cdot^\cdot$ denotes inclusion, $c_x$ denotes conjugation by $x$, and $[V\backslash U/ W]$ is a set of representatives in $G$ for the double cosets $V\backslash U/W$. If we restrict these functors to the meet-semilattice of normal subgroups of $G$, formula \eqref{Mackey/Mackey equation} becomes
\[
M^*([\iota_V^U])\circ M_*([\iota_W^U]) = M_*([\iota_{V\cap W}^ V])\circ\left(\sum_{x\in [V\setminus U/W]} M_*([c_x])\right)\circ M^*([\iota_{V\cap W}^ W]).
\]
\end{exm}

\begin{defn}
\label{def:MackeyfunctorCC}
For $\CP$ a filtered meet-semilattice and a category $\CC$, a pair of functors $(F,G)$ is a \emph{Mackey functor} if $F\colon \CP\to \CC$ is covariant, $G\colon \CP^{op}\to \CC$ is contravariant, $F(i)=G(i)$ for all $i\in \CP$, and for all $j<i,k<i$ there exist 
\[
\alpha(i,j,k)\in \End_\CC(F(j))\text{, }\beta(i,j,k)\in \End_\CC(F(k\wedge j))\text{, and }\gamma(i,j,k)\in \End_\CC(F(k))
\]
such that
\begin{align*}
G(j<i)\circ F(k<i)=&\alpha(i,j,k) \circ F(k\wedge j<j)\circ G(k\wedge j<k)\\ 
=&  F(k\wedge j<j)\circ \beta(i,j,k)\circ G(k\wedge j<k)\\
=&  F(k\wedge j<j)\circ G(k\wedge j<k)\circ\gamma(i,j,k).
\end{align*}
We say that $(F,G)$ has a \emph{quasi-unit} if $\alpha(i,j,j), \gamma(i,j,j), \beta(i,j,j) \in \Aut_\CC(F(j))$ for $j<i$.
\end{defn}

The term \emph{quasi-unit} is borrowed from \cite[5.7 Definition]{JACKOWSKI1992113} and implies acyclicity. Example \ref{exm:ClassicMackeyNormal} restricted to the meet-semilattice of central subgroups is a Mackey functor with $\alpha(U,V,V)$ equal to multiplication by $|U/V|$, so that this Mackey functor has a quasi-unit if and only if these numbers are invertible in the ring $R$.

In addition, for $\CC=R\lmod$, we also define \emph{weak Mackey functor} by dropping the contravariant functoriality and the meet-semilattice constraint; see \cref{def:weakMackeyfunctorAb} for details. The contravariant part of a Mackey functor is then a weak Mackey functor, and we have the following result.

\begin{introtheorem}\label{thm:Mackey_is_acyclic}
Let  $F\colon \CP\to R\lmod$ be a weak Mackey functor with a quasi-unit (or the covariant part of a Mackey functor with a quasi-unit). Then $F$ is $\colim$-acyclic.
\end{introtheorem}

Notice that any functor $H\colon \CC\to \CD$ takes Mackey functors over $\CC$ to Mackey functors over $\CD$. A natural example of such a functor $H$ is given by the homology functor $H_*\colon \SSets \to R\lmod$. In the category of spaces $\SSets$, the previous model category techniques work and, as before, we can drop the contravariant functoriality and define weak Mackey functors over $\SSets$, see Definition \ref{def:weakMackeyfunctorSSets}. We deduce the following topological application involving homotopy colimits and their homology. 

\begin{introtheorem}\label{thm:MackeySSets_nice_hocolim}
Let $F\colon \CP\to \SSets$ be a weak Mackey functor with a quasi-unit (or the covariant part of a Mackey functor with a quasi-unit). Then $\hocolim_\CP F\simeq \colim_\CP F$ and $H_*(\hocolim_\CP F)=\colim H_*(F)$.
\end{introtheorem}

In the language of \cite{D1998},  \cref{thm:MackeySSets_nice_hocolim} asserts that the corresponding integral homological decomposition is \emph{sharp}. This theorem may be applied to the construction of Davis-Januszkiewicz spaces and its generalization to polyhedral products over finite posets, see Examples \ref{exm:Davis-Januszkiewicz} and \ref{exm:KishimotoLevi}. As a second application of pseudo-projectivity, we provide a integral homology decomposition of the classifying space of the Bianchi group $\Gamma_1=\PSL_2(\mathcal{O}_1)$, where $\mathcal{O}_1$ is the ring of integers of the quadratic field $\QQ(\sqrt{-1})$.

\begin{introtheorem}\label{thm:integral_homology_decomposition_Bianchi}
There exists a poset $\CP_1$ of proper subgroups of $\Gamma_1$ and a homotopy equivalence
\[
\hocolim_{U\in \CP_1} B(U)\stackrel{\simeq}\longrightarrow B(\Gamma_1).
\]
\end{introtheorem}

Here, the morphisms in $\CP_1$ consists of the inclusions between the given subgroups. The notion of pseudo-projectivity is employed here to study the fiber of the natural map from the homotopy colimit to $B(\Gamma_1)$. It turns out that the homology of this fiber can be computed as the higher direct limits of an acyclic functor that is pseudo-projective at many but not all objects of $\CP_1$. As a direct consequence of this theorem, we can determine the homology and cohomology of $\Gamma_1$ using the classical Bousfield-Kan spectral sequence, see  \cref{homology_Bianchi_groups}. This contrasts with earlier computations of these (co)homology groups, in which the equivariant cohomology spectral sequence for an action of $\Gamma_1$ on a $2$-complex is employed, see \cite{SchwermerVogtmann}, \cite{Berkove}, \cite{BerkoveIntegral}.

\noindent
\textbf{Notation:} For a poset $\CP$ and a functor $F\colon \CP\to\CC$, we denote by $F(i<j)$ the image by $F$ of the only arrow from $i$ to $j$ and, with some abuse of notation, we define $F(i<i)=1_{F(i)}$. If $I$ and $J$ are subsets of $\CP$, we write ($J<I$) $J\leq I$ if every element of $J$ is (strictly) smaller than some element of $I$. If $\CC$ is co-complete and $i\in \CP$, we denote the natural map from the colimit as $\epsilon\colon \colim_{\CP_{<i}} F\to F(i)$, and we write $[x]$ to denote the class in this colimit with representative $x$. If $\CC=R\lmod$ and $x=\oplus_{i\in \CP} x_i\in \bigoplus_{i\in \CP} F(i)$, we set $\supp(x)=\{i\in \CP\mid x_i\neq 0\}$. \medskip

\noindent
\textbf{Outline of the paper:} 
In Section \ref{sec:model_categories_and_homotopy_theory}, we prove Theorem \ref{thm:pseudo-sii-locally} and, in Section \ref{sec:Mackey_functors}, we prove Theorems \ref{thm:Mackey_is_acyclic} and \ref{thm:MackeySSets_nice_hocolim}. Posets of subgroups and the Bianchi group $\Gamma_1$ are studied in Section \ref{section:poset_of_subgroups}. The dual results are presented without proof in Section \ref{section:dual}. The proof of Theorem \ref{thm:pseudo-sii-locally_over_DCC}  is postponed to Section \ref{sec:prf_main_thm}. 
\medskip

\noindent
{\bf Acknowledgments.} First author supported by Universidad de Málaga grant G RYC-2010-05663, MICINN grant BES-2017-079851, and PID2020-116481GB-I00. Both authors supported by MICINN grant PID2020-118753GB-I00. Second author supported by Junta de Andalucía grant ProyExcel00827.


\section{Model categories and homotopy theory}\label{sec:model_categories_and_homotopy_theory}

In this section, $R$ denotes a commutative ring with unit, and $\CP$ denotes a filtered poset unless stated otherwise.

\begin{thm}\label{thm:Modelcategory_Ch(Ab)}
 There exists a model category on $\Fun(\CP,\Ch(R))$ such that:
 \begin{enumerate}
     \item a natural transformation $\eta\colon F\Rightarrow G$ is a weak equivalence if and only if $\eta_p$ is a weak equivalence for every $p\in \CP$;
     \item a functor $F$ is cofibrant if only if the natural map $\colim_{\CP_{<i}} F\to F(i)$ is injective for every $i\in \CP$; and
     \item  if a functor $F\colon \CP\to R\lmod$ is cofibrant, then it is $\colim$-acylic, i.e., $\colim_n F=0$ for all $n>0$.
 \end{enumerate}
 \end{thm}
 \begin{proof}
We see $\CP$ as Reedy category, see \cite[Section 15]{Hirschhorn2003} for background, by setting $\direct\CP=\CP$ and $\invers{\CP}$ equal to the discrete category $\CP^0$, and we endow $\Ch(R)$ with the injective model category structure, see \cite[Theorem 18.5.4]{MAYPONTO2011}. Hence, the weak equivalences and the cofibrations in $\Ch(R)$ are, respectively, the quasi-isomorphisms and the monomorphisms. Therefore, by the result of D. M. Kan \cite[Theorem 15.3.4]{Hirschhorn2003}, we can equip $\Fun(\CP,\Ch(R))$ with a model category structure satisfying points (1) and (2) of the statement and the following,
\begin{equation*}
\text{ $\eta\colon F\Rightarrow G$ is a fibration if and only if  $\eta_p$ is a fibration in $\Ch(R)$ for every $p\in \CP$.}
\end{equation*}
Now, the adjoint pair, $\colim_\CP\colon \Fun(\CP,\Ch(R))\Leftrightarrow \Ch(R)\colon \Delta$, where $\Delta$ is the diagonal functor, is a Quillen pair as, by the comments above, $\Delta$ preserves acyclic fibrations and fibrations. Thus, by \cite[Lemma 8.5.9]{Hirschhorn2003}, the total left derived functor $\Ho(\Fun(\CP,\Ch(R)))\to \Ho(\Ch(R))$ exists and takes $[F]$ to $[\colim (QF)]$, where ${Q}F$ is a cofibrant replacement of $F$. As there exist enough projectives in $\Fun(\CP,\Ch(R))$, we recover its usual $i$-th left derived functors as ${\colim_i F=H_i(\colim QF)}$. If $F$ is cofibrant, we may choose $QF=F$ so that $QF$ is concentrated in degree $0$ and hence $H_i(\colim QF)=0$ for $i>0$. This proves (3) of the statement.
\end{proof}

\noindent\textbf{\cref{thm:pseudo-sii-locally}}
\textit{Let $\CP$ be a filtered poset and $F\colon \CP\to R\lmod$ be a functor. Then $F$ is cofibrant if and only if $F$ is pseudo-projective. Moreover, in this case, $F$ is $\colim$-acyclic.}\medskip

\begin{proof}
    Let $F\colon \CP\to R\lmod$ be a functor. By \cref{thm:pseudo-sii-locally_over_DCC}, $F$ is pseudo projective if and only if, for every $i\in \CP$, the natural map
    $
    \colim_{\CP_{<i}} F\to F(i)
    $
    is injective. \cref{thm:Modelcategory_Ch(Ab)} completes the proof.
\end{proof}

Next, we prove a version of Theorem \ref{thm:Modelcategory_Ch(Ab)} for the category $\SSets$ of simplicial sets.

\begin{thm}\label{thm:Modelcategory_SSets}
 If $\CP$ is a filtered poset, there exists a model category on $\Fun(\CP,\SSets)$ such that 
 \begin{enumerate}
     \item a natural transformation $\eta\colon F\Rightarrow G$ is a weak equivalence if and only if $\eta_p$ is a weak equivalence for every $p\in \CP$;
     \item a functor $F$ is cofibrant if only if the natural map $\colim_{\CP_{<i}} F\to F(i)$ is injective for every $i\in \CP$; and
     \item a cofibrant functor $F\colon \CP\to \SSets$ satisfies that $\hocolim_\CP F\simeq \colim_\CP F$.
 \end{enumerate}
 \end{thm}
\begin{proof}
We consider $\CP$ as a Reedy category by choosing $\direct\CP=\CP$ and $\invers{\CP}=\CP^0$. In $\SSets$, we consider the usual model category structure \cite[11.1]{Dwyer1995}: the weak equivalences are those simplicial maps whose topological realization is a weak homotopy equivalence, cofibrations are the (degreewise) injections, and the fibrations are the Kan fibrations. Again by  \cite[Theorem 15.3.4]{Hirschhorn2003}, we equip $\Fun(\CP,\SSets)$ with a model category structure satisfying points (1) and (2) of the statement and the following,
\begin{equation*}
\text{ $\eta\colon F\Rightarrow G$ is a fibration if and only if $\eta_p$ is a fibration in $\SSets$ for every $p\in \CP$.}
\end{equation*}
Now, the adjoint pair, $\colim_\CP\colon \Fun(\CP,\SSets)\leftrightarrow \SSets\colon \Delta$, where $\Delta$ is the diagonal functor, is a Quillen pair as $\Delta$ preserves acyclic fibrations and fibrations. By \cite[Lemma 8.5.9]{Hirschhorn2003}, the total left derived functor $\hocolim_\CP\colon \Ho(\Fun(\CP,\SSets))\to \Ho(\SSets))$ exists and maps $[F]$ to $[\colim (QF)]$, where $F\colon \CP\to \SSets$ and $QF$ is its cofibrant replacement. If $F$ is cofibrant we may choose $QF=F$ so that $\hocolim_\CP F\simeq \colim_\CP F$, obtaining (3) of the statement. 
\end{proof}

\section{Weak Mackey functors}\label{sec:Mackey_functors}

In this section, $\CP$ denotes a filtered poset and we weaken Definition \ref{def:MackeyfunctorCC} for $\CC=R\lmod$ and $\CC=\SSets$. We start by defining a certain class of morphisms that commute with a given functor.

\begin{defn}\label{def:F-linearCC}
Let $F\colon \CP \to \CC$ be a functor between a poset $\CP$ and a category $\CC$, and $i\in \CP$. We say that $\alpha\in \End_\CC(F(i))$  ($\alpha\in \Aut_\CC(F(i))$) is \emph{$F$-linear} if for all $j<i$
\[
\alpha\circ F(j<i)=F(j<i)\circ \beta\]
for some $\beta\in \End_\CC(F(j))$ ($\beta\in \Aut_\CC(F(j))$). We denote by $\End^F_\CC(i)$ ($\Aut^F_\CC(i)$) the submonoid (subgroup) of $F$-linear endomorphisms (automorphisms) of $F(i)$.
\end{defn}

Note that the cases $k=j$ and $k<j$ in Definition \ref{def:MackeyfunctorCC} imply that $\alpha(i,i,j)$ is $F$-linear. Moreover, for $\CC=R\lmod$, and $j\not\leq k$, we have $k\wedge j\neq j$ and  $G(j<i)\circ F(k<i)$ lands in $\im_F(j)$. This fact and earlier comments motivate the following definition, where we recall that the meet-semilattice condition is dropped.
 
\begin{defn}\label{def:weakMackeyfunctorAb}
Let $\CP$ be a filtered poset and let $F\colon \CP\to R\lmod$ a covariant functor. We say that $F$ is a \emph{weak Mackey functor} if for all $j<i$  there exists a morphism in $R\lmod$, $G(j<i)\colon F(i)\to F(j)$, such that $G(j<i)\circ F(j<i)=\alpha(i,j)$ with $\alpha(i,j)\in \End^F_{R\lmod}(j)$, and, for $k<i$, $j\not\leq k$, 
\[
\im(G(j<i)\circ F(k<i))\subseteq \im_F(j).
\]
We say that $F$ has a \emph{quasi-unit} if $\alpha(i,j)\in \Aut^F_{R\lmod}(j)$.
\end{defn}

\begin{rmk}
Notice that the covariant part of a Mackey functor over $R\lmod$ is a weak Mackey functor.
\end{rmk}

\noindent\textbf{\cref{thm:Mackey_is_acyclic}}
\textit{Let  $F\colon \CP\to R\lmod$ be a weak Mackey functor with a quasi-unit (or the covariant part of a Mackey functor with a quasi-unit). Then $F$ is $\colim$-acyclic.}\medskip

The proof of this theorem is reduced to prove the following lemma.

\begin{lem}\label{lem:Mackey_Ab_implies_cofibrant}
If $F\colon \CP\to R\lmod$ is a weak Mackey functor with a quasi-unit, then $F$ is cofibrant.
\end{lem}
\begin{proof}
By Theorem \ref{thm:pseudo-sii-locally}, this is equivalent to prove that $F$ is pseudo-proyective. So let $i\in \CP$, $J\subset \CP_{\le i}$, $J$ finite, and $\oplus_{j\in J}x_j\in\bigoplus_{j\in J} F(j)$ such that 
\[
\sum_{j\in J} F(j<i)(x_j)=0.
\]
For $k\in \max J$, we want to prove that $x_k\in \im_F(k)$. Applying the morphism $G(k<i)$ to the equation above we get
\begin{align*}
0&=G(k<i)(\sum_{j\in J} F(j<i)(x_j))=\sum_{j\in J} (G(k<i)\circ F(j<i))(x_j)\\
&=(G(k<i)\circ F(k<i)(x_k)+\sum_{j\in J,\, k\not\leq j} (G(k<i)\circ F(j<i))(x_j),
\end{align*}
where the first summand corresponds to $j=k$. Now, as $F$ is weak Mackey, this summand is equal to $\alpha(i,k)(x_k)$, and every other other summand belongs to $\im_F(k)$ by Definition \ref{def:weakMackeyfunctorAb}, so that
\[
\alpha(i,k)(x_k) =\sum_{l<k} F(l<k)(y_l)
\]
for some elements $y_l\in F(l)$. As $\alpha(i,k)$ is invertible and $F$-linear, we can solve for $x_k$ as follows, 
\[
x_k =\sum_{l<k} (\alpha(i,k)^{-1}\circ F(l<k))(y_l)=\sum_{l<k} (F(l<k)\circ \beta_l)(y_l)
\]
for some automorphisms $\beta_l\in \Aut_{R\lmod}(F(l))$. Hence $x_k\in \im_F(k)$ and we are done.
\end{proof}

By dropping the contravariant functoriality in Definition \ref{def:MackeyfunctorCC}, we obtain the following weaker notion for $\CC=\SSets$.

\begin{defn}
\label{def:weakMackeyfunctorSSets}
For $\CP$ a filtered meet-semilattice and $F\colon \CP\to \SSets$ a covariant functor. We say that $F$ is a \emph{weak Mackey functor} if there exists a collection of morphisms, 
\[
\{G(j<i)\colon F(i)\to F(j)\}_{j<i},
\]
such that for all $j<i,k<i$ there exist 
\[
\alpha(i,j,k)\in \End_\SSets(F(j))\text{ and }\beta(i,j,k)\in \End_\SSets(F(k\wedge j))
\]
with
\begin{align*}
G(j<i)\circ F(k<i)=&\alpha(i,j,k) \circ F(k\wedge j<j)\circ G(k\wedge j<k)\\ 
=&  F(k\wedge j<j)\circ \beta(i,j,k)\circ G(k\wedge j<k).
\end{align*}
We say that $F$ has a \emph{quasi-unit} if $\alpha(i,j,j)\in \Aut_\SSets(F(j))$ for all $j<i$.
\end{defn}

\begin{rmk}\label{rmk:MackeyonSSets}
Notice that the covariant part of a Mackey functor over $\SSets$ is a weak Mackey functor and that $\alpha(i,i,j)$ in Definition \ref{def:weakMackeyfunctorSSets} is $F$-linear. Moreover, for any functor $H\colon \SSets\to R\lmod$ and any weak Mackey functor $F\colon \CP\to \SSets$, the composition $H\circ F\colon \CP\to R\lmod$ is a weak Mackey functor.
\end{rmk}

\begin{rmk}\label{rmk:Mackeyfunctor_is_monic}
If $F\colon \CP\to \CC$ is a weak Mackey functor for  $\CC=R\lmod$ or $\CC=\SSets$ and it has a quasi-unit, then $F$ is \emph{monic}, i.e., $F(i<j)$ is injective for all $i<j$. This can be easily seen by considering the case $k=j$ in Definition \ref{def:weakMackeyfunctorAb} or \ref{def:weakMackeyfunctorSSets}. 
\end{rmk}

\noindent\textbf{\cref{thm:MackeySSets_nice_hocolim}}
\textit{Let $F\colon \CP\to \SSets$ be a weak Mackey functor with a quasi-unit (or the covariant part of a Mackey functor with a quasi-unit). Then $\hocolim_\CP F\simeq \colim_\CP F$ and $H_*(\hocolim_\CP F)=\colim H_*(F)$.}\medskip

We first prove the next auxiliary lemma.

\begin{lem}
\label{lem:MackeyfunctorSSets_is_cofibrant}
If $F\colon \CP\to \SSets$ is a weak Mackey functor with a quasi-unit, then $F$ is cofibrant.
\end{lem}
\begin{proof}
Let $i\in \CP$. We want to check that the map of simplicial sets $\epsilon\colon\colim_{\CP_{<i}} F\to F(i)$ is (degreewise) injective. So choose a degree $n$ and denote, by abusing notation, the $n$-simplices of $F$ and $G$ by the same letters, respectively. Recall that 
\[
\colim_{\CP_{<i}} F=\bigcup_{j<i} F(j)/\sim,
\]
where $x\sim F(k<j)(x)$ for $k<j<i$. We denote by $[x]$ the image by $F(j)\to\colim_{\CP_{<i}} F$ of $x\in F(j)$. The natural map $\epsilon$ is defined by $\epsilon([x])=F(j<i)(x)$ if $x\in F(j)$. Assume that $y=\epsilon([x])=\epsilon([x'])$ for $x\in F(j)$, $x'\in F(j')$. Applying the definition of Mackey functor we have then
\begin{align*}
G(j<i)(y)&=G(j<i)(F(j<i))(x))=\alpha(i,j,j)(x),\text{ and}\\
G(j<i)(y)&=G(j<i)(F(j'<i))(x'))=\alpha(i,j,j')(F(j\wedge j'<j)(G(j\wedge j'<j')(x'))).
\end{align*}
Hence we have
\begin{align*}
\alpha(i,j,j)(x)&=\alpha(i,j,j')(F(j\wedge j'<j)(G(j\wedge j'<j')(x')))\\
&=F(j\wedge j'<j)(\beta(i,j,j')(G(j\wedge j'<j')(x')))
\end{align*}
and, as $\alpha(i,j,j)$ is an $F$-linear automorphism, we get to
\[
x=F(j\wedge j'<j)(x'')
\]
for some $x''\in F(j\wedge j')$. In particular, 
\[
F(j\wedge j'<i)(x'')=F(j<i)(F(j\wedge j'<j)(x''))=F(j<i)(x)=y.
\]
Switching the roles of $x$ and $x'$ we obtain $x'''\in F(j\wedge j')$ with ${F(j\wedge j'<i)(x''')=y}$. As $F(j\wedge j'<i)$ is injective, see Remark \ref{rmk:Mackeyfunctor_is_monic}, we get that $x''=x'''$. So we have ${x\sim x''=x'''\sim x'}$, i.e., $[x]=[x']$.
\end{proof}

\begin{proof}[Proof of \cref{thm:MackeySSets_nice_hocolim}]
By Remark \ref{rmk:MackeyonSSets}, we may assume that $F$ is a weak Mackey functor. Then the fact that $\hocolim_\CP F\simeq \colim_\CP F$ is a consequence of Theorem \ref{thm:Modelcategory_SSets}(3) and Lemma \ref{lem:MackeyfunctorSSets_is_cofibrant}. In addition, by the same remark, for all $n\geq 0$, $H_n\circ F$ is a weak Mackey functor with codomain $R\lmod$, where $H_n$ is homology in degree $n$. Applying then Theorem \ref{thm:Mackey_is_acyclic} we get that $\colim_i (H_n\circ F)=0$ for $i>0$. Then the Bousfield-Kan homology spectral sequence for the homotopy colimit $\hocolim_\CP F$ collapses and gives the statement that $H_*(\hocolim_\CP F)=\colim H_*(F)$.
\end{proof}
\begin{exm}
\label{exm:Davis-Januszkiewicz}
One example of Mackey functors for posets are twin functors for the face poset of a simplicial complex. For any simplicial complex $K$ and its face poset $\CAT(K)$, the twin functors $X^K\colon \CAT(K)\to \Top$ and $X_K\colon \CAT(K)^\op\to \Top$ in \cite[Definition 2.6]{NotbohmRay_DJ} are a Mackey functor in the sense of Definition \ref{def:MackeyfunctorCC} with a quasi-unit and with $\alpha(i,j,k)=1_{F(j)}$ for all $j<i,k<i$ of $\CAT(K)$. Hence, \cite[Equation 2.8]{NotbohmRay_DJ},
\[
\hocolim X^K\to \colim X^K\mbox{ is a homotopy equivalence},
\]
is a direct consequence of \cref{thm:MackeySSets_nice_hocolim}.
\end{exm}


\section{Poset of subgroups}\label{section:poset_of_subgroups}

In this section, $G$ denotes a group and $\CP$ the inclusion poset of a collection of subgroups of $G$, subject to the following condition,
\begin{equation}\label{equ:hypothesis_on_subgroup_poset}
1\in \CP\text{ and }\{U\hookrightarrow G\}_{U\in \CP}\text{ is the cone to }G\cong\colim_{U\in \CP} U.
\end{equation}
Under these conditions,  the arguments in the proof \cite[Theorem 5.1]{DiazRamos2009} give a fibration
\begin{equation}\label{equ:fibration_to_BGamma}
F\to \hocolim_{U\in \CP} B(U)\to B(G),
\end{equation}
where $F$ is simply connected and $H_n(F;\ZZ)=\colim_{n-1} H$ for $n\geq 2$. The functor ${H\colon \CP\to \Ab}$ takes the following value on $U\in \CP$,
\[
H(U)=\{\sum_{g\in G} n_g\cdot g\in \ZZ[G]\mid \sum_{u\in U} n_{ug}=0\text{ for all $g\in G$}\},
\]
and takes morphisms to inclusions. Note that if $H$ is acyclic then $F\simeq *$ and we have a homotopy equivalence between the homotopy colimit and $B(G)$ in \eqref{equ:fibration_to_BGamma}. The following result gives conditions on the subgroups in the ray $P_{\le U}$ that are sufficient for $H$ to be pseudo-projective at $U$.

\begin{prop}\label{prop:H_pseusoprojective_at_U}
If for every finite subset $J\subset \CP_{\le U}$, and setting $M=\max J$, the subgroup $\langle V\mid V\in M\rangle$ is the free product of the subgroups $\{V\}_{V\in M}$ with amalgamated subgroup $\bigcap_{V\in M} V\in P$, then $H$ is pseudo-projective at $U$.
\end{prop}
\begin{proof}
Given $\oplus_{V\in J}x_V\in \bigoplus_{V\in J} H(V)$ with $\sum_{V\in J} x_V=0$, we need to show that
\begin{equation}
x_V\in \im_H(V)\text{ for every }V\in M.
\end{equation}
As $J\leq M$, we may rewrite the sum $\sum_{V\in J} x_j$ as $\sum_{V\in M} x'_V$, where $x'_V\in H(V)$ for each $V\in M$. Note that $x_V -x'_V\in \im_H(V)$ for all $V\in M$ so that it is enough to show that 
\begin{equation}\label{thmE:thesis1_V}
x'_V\in \im_H(V)\text{ for every }V\in M
\end{equation}
subject to the condition
\begin{equation}\label{thmE:sum'_is_zero_V}
\sum_{V\in M} x'_V=0.    
\end{equation}

To that aim, we follow \cite[Theorem 1]{L1970} adapted to an arbitrary number $|M|$ of subgroups and we spell out the details here: We may write
\[
x'_V=\sum_l (v_V^l-1)\epsilon_V^l g_V^l,
\]
where $v_V^l\in V$, $\epsilon_V^l=\pm 1$ and $g_V^l\in G$. Define $K=\langle V\mid V\in M\rangle$. Note that, as $V\leq K$ for all $V\in J$, Equation \eqref{thmE:sum'_is_zero_V} is true on every right coset $K\setminus G$ and we may assume that the elements $g_V^l$ belong to $K$. As $K$ is the free product of the groups $\{V\}_{j\in M}$ with amalgamated subgroup $\bigcap_{V\in M} V=W\in \CP$, by \cite[Equation (6), p.30]{K1956}, every element $g$ of $K$ either belongs to $W$ or may be written as 
\[\
g=g_1g_2\ldots g_k
\]
with $g_l\in V\setminus W$ for some $V\in M$ for all $1\leq l\leq k$, and such that adjacent elements $g_l$, $g_{l+1}$ lie in different subgroups $V$. The number $k$ in the expression above is the length $l(g)$ and we set $l(g)=0$ for $g\in W$. To prove that Equation \eqref{thmE:thesis1_V} holds, we do induction on 
\[
m=\max_{V\in M,l}\{ l(g_V^l)\}\text{ and on }|N|\text{ with }N=\{g_V^l\mid l(g_V^l)=m\}.
\]
For the base case of the induction, i.e., for $m=0$, we have that $g_V^l\in W$ for all $V$ and $l$. Thus, we may write, for each $V\in M$,
\[
x'_V=\sum_{v_V\in V} n_{v_V}\cdot v_V
\]
with $\sum_{v_V\in V} n_{v_V}=0$. For any $V\in M$, Equation \eqref{thmE:sum'_is_zero_V} shows that every element $v_V\in V$ with $n_{v_V}\neq 0$ belongs at least to some other subgroup $V'$ with $V\neq V'\in M$. By the hypothesis of this proposition, we have that $V\cap V'\in \CP$ so that  $v_V\in V\cap V'$ and $v_V-1\in H(V\cap V')\leq \im_H(V)$. Thus, 
\begin{align*}
x'_V&=\sum_{v_V\in V} n_{v_V}\cdot v_V=\sum_{v_V\in V} n_{v_V}\cdot (v_V - 1 + 1 )\\
&= \sum_{v_V\in V} n_{v_V}\cdot (v_V - 1) + (\sum_{v_V\in V} n_{v_V})\cdot 1 = \sum_{v_V\in V} n_{v_V}\cdot (v_V - 1)\in \im_H(V).
\end{align*}

For the general case of the induction with $m\geq 1$, we may assume that $g_V^l\in N$ has a decomposition $g_1g_2\ldots g_k$ with $g_1\notin V$. For, if $g_1\in V$, then 
\[
(v_V^l-1)\epsilon_V^lg_V^l=(v_V^lg_1-1)\epsilon_V^l g_2\ldots g_k-(g_1-1)\epsilon_V^lg_2\ldots g_k
\]
and we have found an expression for $x'_V$ with fewer summands of length $m$. We may also assume, for $g_V^l\in N$, that $v_V^l\notin W$ as, otherwise, we may write
\[
z_V=(v_V^l-1)\epsilon_V^lg_V^l=(v_V^lg_1-1)\epsilon_V^l g_2\ldots g_k-(g_1-1)\epsilon_V^lg_2\ldots g_k
\]
where $g_1\in V'$ and $V\neq V'\in M$. Then we may rewrite Equation \eqref{thmE:sum'_is_zero_V} as
\[
\sum_{V''\in M\setminus \{V,V'\}} x'_{V''}+(x'_V-z_V)+(x'_{V'}+z_V)=0,
\]
where again there are fewer summand of length $m$. So, by induction, $x'_V-z_V$ and $x'_{V'}+z_V$ belong to $\im_H(V)$ and $\im_H(V')$ respectively. Thus , we are done if we show that $z_V\in\im_H(V)\cap \im_H(V')$, and this is clear under the assumption that $v^l_V\in W$.

Finally, for any $V\in M$, the expression $y_V=\sum_{l\mid g_V^l\in N} v_V^l\epsilon_V^lg_V^l$ is the sum of all the terms in \eqref{thmE:sum'_is_zero_V} which have length $m+1$ and first symbol in $V$. So $y_V=0$ and there exists a pair of elements $l,l'$ with $g_V^l,g_V^{l'}\in N$ and such that $v_V^l\epsilon_V^lg_V^l+v_V^{l'}\epsilon_V^{l'}g_V^{l'}=0$, $(v_V^{l'})^{-1}v_V^l\epsilon_V^lg_V^l=-\epsilon_V^{l'}g_V^{l'}$, and
\[
(h_V^l-1)\epsilon_V^lg_V^l+(v_V^{l'}-1)\epsilon_V^{l'}g_V^{l'}=((v_V^{l'})^{-1}v_V^l-1)\epsilon_V^lg_V^l.
\]
Thus, we have an expression for each $x'_V$ with fewer summands of length $m$.
\end{proof}

In the rest of this section, we show that for the Bianchi group $\Gamma_1$ and a suitable chosen poset, the functor $H$ is acyclic but not pseudo-projective. In fact, consider the following presentation of $\Gamma_1$, see \cite{MR0372059},
\[
\Gamma_1=\langle a,b,c,d\mid a^3=b^2=c^3=d^2=(ac)^2=(ad)^2=(bc)^2=(bd)^2=1\rangle,
\]
as well as the following poset $\CP_1$ of subgroups of $\Gamma_1$,
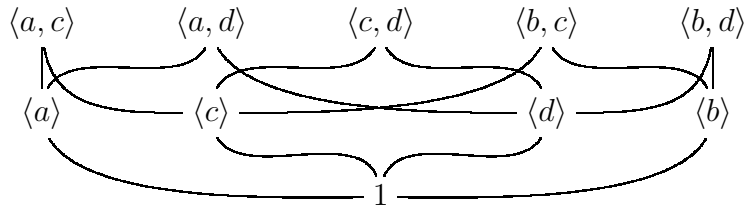
\[
\xymatrix@C=30pt@R=15pt{
{\langle a,c\rangle}& {\langle a,d\rangle}& {\langle c, d\rangle}& {\langle b,c\rangle}& {\langle b,d\rangle}\\
{\langle a\rangle}\ar@{-}@(u,d)[u]\ar@{-}@(u,d)[ru]& {\langle c\rangle}\ar@{-}@(l,d)[lu]\ar@{-}@(u,d)[ru]\ar@{-}@(r,d)[rru]&& {\langle d\rangle}\ar@{-}@(u,d)[lu]\ar@{-}@(r,d)[ru]\ar@{-}@(l,d)[llu]& {\langle b\rangle}\ar@{-}@(u,d)[lu]\ar@{-}@(u,d)[u]\\
&&1\ar@{-}@(l,d)[llu]\ar@{-}@(u,d)[lu]\ar@{-}@(u,d)[ru]\ar@{-}@(r,d)[rru]&&
}
\]
\bigskip

\noindent\textbf{\cref{thm:integral_homology_decomposition_Bianchi}}
\textit{There is a  homotopy equivalence,
\[
\hocolim_{U\in \CP_1} B(U)\stackrel{\simeq}\longrightarrow B(\Gamma_1).
\]}
\begin{proof}
We need to show that $H$ is acyclic, i.e., that $\colim_1 H=0$. The usual chain complex \cite{GabrielZisman} gives that $\colim_1 H$ consists of the tuples
\[
(x_a,x'_a,x_b,x'_b,x_c,x'_c,x''_c,x_d,x'_d,x''_d),
\]
where $x^\cdot_y\in H(\langle y\rangle)$, such that 
\begin{align}\label{equ:kernel_in_C1_of_Gamma_1}
x_a+x'_a&=0\text{, }x_b+x'_b=0\text{, }x_c+x'_c+x''_c=0\text{, }x_d+x'_d+x''_d=0\text{, and }\\
x_a+x_c&=0\text{, }x'_a+x_d=0\text{, }x'_c+x'_d=0\text{, }x_b+x''_c=0\text{, }x'_b+x''_d=0.\nonumber
\end{align}
As $\langle c,d\rangle=\langle c\rangle\ast \langle d\rangle \cong \PSL_2(\ZZ)$, by Proposition \ref{prop:H_pseusoprojective_at_U} we have that $H$ is pseudo-projective at this subgroup. As $H(1)=0$, we have $\im_H(\langle c\rangle)=\im_H(\langle d\rangle)=0$, and hence
\begin{equation}\label{equ:Hpp_at_cd}
C+D=0\Rightarrow C=D=0\text{ for any }C\in H(\langle c\rangle)\text{ and any }D\in H(\langle d\rangle).
\end{equation}
From \eqref{equ:kernel_in_C1_of_Gamma_1}, it is easy to deduce that
\[
x_c+x_d=x'_c+x'_d=x''_c+x''_d=0,
\]
and hence, from \eqref{equ:Hpp_at_cd}, we get to $x_c=x_d=x'_c=x'_d=x''_c=x''_d=0$. Using \eqref{equ:kernel_in_C1_of_Gamma_1} again, we have that $x_a=x'_a=x_b=x'_b=0$ too, i.e., $\colim_1 H=0$.
\end{proof}

\begin{exm}\label{homology_Bianchi_groups}
The Bousfield-Kan spectral sequence \cite[XII.5.7]{BK1972} of the homotopy colimit in \cref{thm:integral_homology_decomposition_Bianchi} converges to $H_*(B(\Gamma_1);\ZZ)$ and has second page
\[
E^2_{p,q}=\colim_p H_q(B(\cdot);\ZZ).
\]
In addition, we have $\langle a,c\rangle\cong A_4$, $\langle a,d \rangle\cong\langle b,c\rangle\cong S_3$, and $\langle b,d\rangle\cong K$, where $A_4$ is alternating on $4$ letters, $S_3$ is symmetric on $3$ letters and $K$ is the Klein group. We refer the reader to \cite[Section 4]{SchwermerVogtmann} for the homology of these groups.

For $q=0$ we get a constant diagram over a contractible poset, so $E^2_{p,0}=\ZZ$ for $p=0$ and $E^2_{p,0}=0$ for $p>0$. For $q$ even, $q>0$, $H_q(B(\cdot);\ZZ)$ takes values different from $0$ only on $\langle a, c\rangle$ and $\langle b,d\rangle$. Thus, 
\begin{equation}\label{equ:all2torsion}
\colim_0 H_q(B(\cdot);\ZZ)=H_q(A_4)\oplus H_q(K)=(\ZZ/2)^l\text{ for some }l,
\end{equation}
and $E^2_{p,q}=0$ for $p>0$. For $q$ odd, we have that
\begin{equation}\label{equ:no3torsion}
\colim_1 H_q(B(\cdot);\ZZ)\text{ is $3$-torsion.}
\end{equation}
In fact, if $q\equiv 1\pmod 4$, the functor $H_q(B(\cdot);\ZZ)$ can be depicted as follows, 
\[
\xymatrix@C=30pt@R=15pt{
H_q(A_4)& \ZZ/2& \ZZ/3\oplus \ZZ/2& \ZZ/2& (\ZZ/2)^\frac{q+3}{2}\\
\ZZ/3\ar@{-}@(u,d)[u]\ar@{-}@(u,d)[ru]&\ZZ/3\ar@{-}@(l,d)[lu]\ar@{-}@(u,d)[ru]\ar@{-}@(r,d)[rru]&& \ZZ/2\ar@{-}@(u,d)[lu]\ar@{-}@(r,d)[ru]\ar@{-}@(l,d)[llu]& \ZZ/2.\ar@{-}@(u,d)[lu]\ar@{-}@(u,d)[u]
}
\]
It is easily checked that this functor is pseudo-projective at $\langle c,d\rangle$ and $\langle b,d\rangle$. Then, a similar computation to that in the proof of \cref{thm:integral_homology_decomposition_Bianchi}, yields that there is no $2$-torsion in $\colim_1 H_q(B(\cdot);\ZZ)$. For $q\equiv 3\pmod 4$, the computation is easier as the functor turns out to be pseudo-projective at $\langle a,d\rangle$, $\langle c,d\rangle$, $\langle b,c\rangle$,  and $\langle b,d\rangle$.

Summing up, $E^2_{p,q}$ is non-zero only at $(p,q)\in \{(0,k),(1,2k+1)\}_{k\geq 0}$. In particular, the spectral sequence collapses at $E_2$. In addition, Equations \eqref{equ:all2torsion} and \eqref{equ:no3torsion} ensure that we can solve the extension problems and then get
\[
H_{n}(B(\Gamma_1);\ZZ)\cong \begin{cases} \ZZ,&\text{if $n=0$,}\\
E^2_{0,n},&\text{if $n$ is odd,}\\
E^2_{0,n}\oplus E^2_{1,n-1},&\text{if $n$ is even, $n>0$.}
\end{cases}
\]
Explicit formulae for this result can be found in \cite[Theorem 5.5]{SchwermerVogtmann}
\end{exm}

\section{Dual}\label{section:dual}
We devote this section to dualize the previous results and we do not include the corresponding proofs. For a filtered poset $\CP$, it is immediate to dualize \cref{thm:Modelcategory_Ch(Ab)} to obtain a model category on $\Fun(\CP^\op,\Ch(R))$ such that a fibrant functor is $\lim$-acyclic and ${G\colon \CP^\op\to R\lmod}$ is fibrant if and only if
\begin{equation}\label{equ:fibrant_to_Ch(Ab)}
G(i)\to \lim _{\CP_{<i}}G\text{ is an epimorphism for every $i\in \CP$.}
\end{equation}
This condition is equivalent, over a filtered poset $\CP$, to the following notion.

\begin{defn}
\label{def:pseudo-injective}
A functor $G\colon \CP^\op\to R\lmod$ over a $\DCC$ poset is \emph{pseudo-injective} at $i\in \CP$ if, for every subset $J\subset\CP_{\le i}$, and elements $x_j\in \ker_G(j)=\bigcap_{k<j}\ker(G(k<j))$ with $j\in J$, there exists $x\in G(i)$ with $G(j<i)(x)=x_j$ for every $j\in \max J$. We say that $G$ is \emph{pseudo-injective} if $G$ is pseudo-injective at $i\in \CP$ for every $i\in \CP$. 
\end{defn}

\noindent\textbf{Theorem B*.}
\textit{
Let $\CP$ be a filtered poset, and $G\colon \CP^\op\to R\lmod$ be a functor. Then $G$ is fibrant if and only if $G$ satisfies \eqref{equ:fibrant_to_Ch(Ab)} if and only if $G$ is pseudo-injective. Moreover, in this case, $F$ is $\lim$-acyclic.
}

\begin{rmk}
The equivalence between \eqref{equ:fibrant_to_Ch(Ab)} and pseudo-injectiveness holds for a $\DCC$ poset if the functor takes Artinian $R$-modules as values. We do not know whether it is true for arbitrary functors over $\DCC$ posets.
\end{rmk}

\begin{defn}
\label{def:G-linearCC}
Let $G\colon \CP^\op \to \CC$ be a functor between a poset $\CP$ and a category $\CC$, and $i\in \CP$. We say that $\gamma\in \End_\CC(G(i))$  ($\gamma\in \Aut_\CC(G(i))$) is \emph{$G$-linear} if for all $j<i$
\[
G(j<i)\circ\gamma = \beta \circ G(j<i)\]
for some $\beta\in \End_\CC(G(j))$ ($\beta\in \Aut_\CC(G(j))$). We denote by $\End^G_\CC(i)$ ($\Aut^G_\CC(i)$) the submonoid (subgroup) of $G$-linear endomorphisms (automorphisms) of $G(i)$.
\end{defn}

\begin{defn}\label{def:weakMackeyfunctorAb_Pop} 
 Let $\CP$ be a filtered poset and let $G\colon \CP^\op\to R\lmod$ a contravariant functor. We say that $G$ is a \emph{weak Mackey functor} if for all $j<i$  there exists a morphism in $R\lmod$, $F(j<i)\colon G(j)\to G(i)$, such that $G(j<i)\circ F(j<i)=\alpha(i,j)$ with $\alpha(i,j)\in \End^G_{R\lmod}(j)$, and, for $k<i$, $j\not< k$, 
\[
\ker_G(k)\subseteq \ker(G(j<i)\circ F(k<i))
\]
We say that $G$ has a \emph{quasi-unit} if $\alpha(i,j)\in \Aut^G_{R\lmod}(j)$.   
\end{defn}

\begin{rmk}
Notice that the contravariant part of a Mackey functor over $R\lmod$ is a weak Mackey functor.
\end{rmk}

\noindent\textbf{Theorem C*.}\textit{
Let $G\colon \CP^{op}\to R\lmod$ be a weak Mackey functor with a quasi-unit (or the contravariant part of a Mackey functor with a quasi-unit). Then $G$ is fibrant and $\lim$-acyclic.
}\medskip

\begin{rmk}
If $F\colon \CP\to \SSets$ is the covariant part of a Mackey functor (or a weak Mackey functor) with quasi-unit, and   $H\colon \SSets^\op \to R\lmod$ is a contravariant functor, then $H\circ F\colon \CP^\op\to R\lmod$ is a contravariant weak Mackey functor.
\end{rmk}

\begin{exm}\label{exm:KishimotoLevi}
For a finite poset $\CP$ with an initial object, a notion of functor $G\colon \CP^{op}\to \Ab$ with a \emph{lower factoring section} is given in \cite[Definition 2.5]{KishimotoLevi}. If $\CP$ is filtered, it is immediate that such a functor is a weak Mackey functor according to Definition \ref{def:weakMackeyfunctorAb_Pop}, taking as quasi-unit $\alpha(i,j)=1_{F(j)}$. Then \cite[Theorem A]{KishimotoLevi} is a direct consequence of Theorem C*.
\end{exm}


\section{The proof of Theorem \ref{thm:pseudo-sii-locally_over_DCC}}\label{sec:prf_main_thm}

Let $\CP$ be a poset, and $F\colon \CP\to R\lmod$ be a functor. We say that $F$ is \emph{cofibrant} at $i\in \CP$ if the natural map
\[
\colim_{\CP_{<i}}F\to F(i)
\]
is injective. We say that $F$ is cofibrant if it is cofibrant at $i$ for every $i\in \CP$. Note that we have abused notation and termed some functors \emph{cofibrant} without any reference to a model category structure. If $\CP$ is filtered, then cofibrant has the usual meaning because of Theorem \ref{thm:Modelcategory_Ch(Ab)}(2). \medskip

\noindent
\textbf{Theorem \ref{thm:pseudo-sii-locally_over_DCC}}\textit{
Let $\CP$ be a DCC poset and $F\colon \CP\to R\lmod$ be a functor. Then $F$ is cofibrant if and only if $F$ is pseudo-projective.
}\medskip

We divide the proof into several lemmas.

\begin{lem}
\label{lem:sequence_aux_lem}
Let $\CP$ be a DCC poset, $i\in \CP$  and $F:\CP\to R\lmod$ be a functor such that $F$ is pseudo-projective at $i$.
Let $x=\oplus_{j<i}x_j\in \bigoplus_{j<i}F(j)$ satisfy
\[
\sum_{j<i}F(j<i)(x_j)=0.
\]
Then, there is a sequence $\{x^n\}_{n\geq 0}$, $x^n=\bigoplus_{j<i}x_j^n\in \bigoplus_{j<i}F(j)$, with $x^0=x$, 
\[
\sum_{j<i}F(j<i)(x^n_j)=0\text{, }x^{n+1}-x^n=\sum_{k<j\in \max\supp(x^n)} y_{k,j}\oplus-F(k<j)(y_{k,j}), 
 \]
 $[x^{n+1}]=[x^n]$ in $\colim_{j<i}F(j)$, and $\supp(x^{n+1})<\supp(x^n)$, for any $n\geq 0$, where ${y_{k,j}\in F(k)}$. In addition, there exists $N>0$ such that $x_j^n=0$ for all $j<i$ if $n\geq N$.
\end{lem}
\begin{proof}
This is a finer reformulation of \cite[Lemma 2.3]{DiazRamos2009} and we provide details. We define $x^{-1}=0$ and work by induction on $n\geq 0$ assuming that $x^n$ satisfying the properties in the statement has been already constructed. Then, as $\sum_{j<i}F(j<i)(x^n_j)=0$ and $F$ is pseudo-projective at $i$, for every $j\in \max\supp(x^n)$ we have that $x^n_j\in \im_F(j)$, i.e., there exists $\oplus_{k<j} y_{k,j}\in \bigoplus_{k<j} F(k)$ such that
\[
x^n_j=\sum_{k<j} F(k<j)(y_{k,j}).
\]
For every pair $(k,j)$ with $k<j<i$, we set
\[
x_{k,j}=\begin{cases}
y_{k,j}& \mbox{if }k<j\in \max \supp(x^n),\\
x^n_j& \mbox{if }k=j\not\in \max \supp(x^n),\\
0& \mbox{otherwise,}
\end{cases}
\]
and we define $x^{n+1}=\bigoplus_{j<i}x_j^{n+1}$ by
\[
x^{n+1}_j=\sum_{k\geq j}x_{j,k}.
\]
Then 
\begin{equation}\label{equ:long_computation}
\sum_{j<i}F(j<i)(x^{n+1}_j)=\sum_{j<i}F(j<i)(\sum_{k\geq j}x_{j,k})=\sum_{j\leq k<i}F(j<i)(x_{j,k}).
\end{equation}
In this last sum, if $j=k\notin \max \supp(x^n)$, the corresponding summand is $F(j<i)(x^n_j)$. The rest of summands we can be reordered as follows,
\begin{scriptsize}
\[
\sum_{\substack{j<k<i\\k\in \max \supp(x^n)}} F(j<i)(y_{j,k})=\sum_{\substack{k<i\\k\in \max \supp(x^n)}} F(k<i)\big(\sum_{j<k}F(j<k)(y_{j,k})\big)=\sum_{\substack{k<i\\k\in \max \supp(x^n)}} F(k<i)(x^n_k).
\]
\end{scriptsize}
Hence the sum \eqref{equ:long_computation} equals $\sum_{j<i}F(j<i)(x^n_j)$, and this is $0$ by hypothesis. From the construction above it easily follows that 
\[
x^{n+1}-x^n=\sum_{k<j\in \max\supp(x^n)} y_{k,j}\oplus-F(k<j)(y_{k,j}),
\]
and, from here, it is clear that $[x^{n+1}]=[x^n]$ in $\colim_{j<i}F(j)$ and ${\supp(x^{n+1})<\supp(x^n)}$. From this latter condition and \cite[Lemma 2.6]{DiazRamos2009} we obtain $N>0$ with the stated property.
\end{proof}

\begin{lem}
Let $\CP$ be a DCC poset, $F\colon \CP\to R\lmod$, and $i\in \CP$. If $F$ is pseudo-projective at $i$, then $F$ is cofibrant at $i$.
\end{lem}
\begin{proof}
Let $\epsilon\colon \colim_{\CP_{<i}}F\to F(i)$ be the corresponding natural map and 
${[x]\in \ker(\epsilon)}$ with $x\in \bigoplus_{j<i} F(j)$. By Lemma \ref{lem:sequence_aux_lem}, there exists a sequence $\{x^n\}_{n\geq 0}$ with $x^n\in \bigoplus_{j<i} F(j)$ such that $x^0=x$, $[x^{n+1}]=[x^n]$ and $x^N=0$ for $N$ big enough. Hence $[x^0]=[x^N]=[0]=0$ and the Lemma is proven.
\end{proof}
\begin{lem}
Let $\CP$ be a $\DCC$ poset, $F\colon \CP\to R\lmod$, and $i\in \CP$. If $F$ is cofibrant at $j$ for every $j\le i$, then $F$ is pseudo-projective at $j$ for every $j\le i$.
\end{lem}
\begin{proof}
Since $\CP_{\le p}$ is a $\DCC$ poset, we proceed by induction. If $j\le i$ is minimal in $\CP$ then $F$ is pseudo-projective at $i$ by definition. Thus, consider now $j\leq i$ such that $F$ is pseudo-projective at $k$ for all $k<j$. We show that $F$ is pseudo-projective at $j$ too. So let $x=\oplus_{k<j}x_k\in \bigoplus_{k<j}F(k)$ be such that
\[
\sum_{k<j}F(k<j)(x_k)=0.
\]
This is equivalent to that $\epsilon([x])=0$ for the natural map $\epsilon\colon \colim_{k<j} F\to F(j)$. By hypothesis, $F$ is cofibrant at $j$, and hence $[x]=0$. In turn, this equality is equivalent to the existence of elements $y_{l,k}\in F(l)$ for $l<k<j$ such that finitely many of them are different from zero and with
\begin{equation}\label{eq:x_k_in_colim}
x=\oplus_{k<j}x_k=\sum_{l<k<j} y_{l,k}\oplus -F(l<k)(y_{l,k}),
\end{equation}
which implies that, for any $k<j$, 
\begin{equation}\label{equ:long_computation3}
x_k=\sum_{k<l}y_{k,l}-\sum_{l<k}F(l<k)(y_{l,k}).
\end{equation}
Let $K=\{k\in \CP_{<j}\mid \exists\ l<k\mbox{ with }y_{l,k}\not=0\}$. We are about to show that we can choose the elements $y_{l,k}$'s appearing in \eqref{eq:x_k_in_colim} subject to the constraint that $\max K\subseteq \supp(x)$. Thus let $m\in \max K\setminus \supp(x)$. Then 
\begin{equation}\label{equ:long_computation4}
x_m=0=-\sum_{l<m}F(l\to m)(y_{l,m}). 
\end{equation}
We can rewrite Equation \eqref{eq:x_k_in_colim} as follows,
\begin{align*}
x=&\sum_{l<m<j} y_{l,m}\oplus -F(l<m)(y_{l,m})+\sum_{\substack{l<k<j\\k\neq m}} y_{l,k}\oplus -F(l<k)(y_{l,k})\\
=&\Big(\oplus_{l<m}y_{l,m}\Big)-\Big(\oplus_m \sum_{l<m}F(l<m)(y_{l,m})\Big)+\sum_{\substack{l<k<j\\ k\neq m}}y_{l,k}\oplus-F(l< k)(x_{l,k}).
\end{align*}
which, by Equation \eqref{equ:long_computation4}, we can simplify to
\begin{equation}\label{eq:x_k_in_colimbis}
x=y+\sum_{\substack{l<k<j\\ k\neq m}}y_{l,k}\oplus-F(l< k)(y_{l,k}),
\end{equation}
where $y=\oplus_{l<m}y_{l,m}\in \oplus_{l<m} F(l)$. As $F$ is pseudo-projective at $m<j$ by induction hypothesis, we may apply Lemma \ref{lem:sequence_aux_lem} to the element $y$ to obtain a sequence of elements $\{y^n\}_{n\geq 0}$ such that $y^0=y$, $[y^{n+1}]=[y^n]$ for all $n\geq 0$, and $y^N=0$ for $N$ big enough. In addition, as $\supp(y^{n+1})<\supp(y^n)$ and $\supp(y)<\{m\}$, we obtain that  
\begin{equation}\label{equ:long_computation2}
y^n-y^{n+1}=\sum_{l<k\in \max\supp(y^n)} z_{l,k}\oplus-F(l<k)(z_{l,k}) =\sum_{l<k<m} z_{l,k}\oplus-F(l<k)(z_{l,k})
\end{equation}
for elements $z_{l,k}\in F(l)$. Define $y^0_{l,k}=y_{l,k}$, assume by induction that 
\begin{equation}\label{eq:x_k_in_colimbisbis}
x=y^n+\sum_{\substack{l<k<j\\ k\neq m}}y^n_{l,k}\oplus-F(l< k)(y^n_{l,k}),
\end{equation}
for elements $y^n_{l,k}\in F(l)$, and note that this holds for $n=0$ by Equation \eqref{eq:x_k_in_colimbis}. For the induction step, we may write
\[
x=y^{n+1}-y^{n+1}+y^n+\sum_{\substack{l<k<j\\ k\neq m}}y^n_{l,k}\oplus-F(l< k)(y^n_{l,k})=y^{n+1}+\sum_{\substack{l<k<j\\ k\neq m}}y^{n+1}_{l,k}\oplus-F(l< k)(y^{n+1}_{l,k}),
\]
for  elements $y^{n+1}_{l,k}\in F(l)$, where in the last equality we have employed Equation \eqref{equ:long_computation2}. Hence, for $n=N$, Equation \eqref{eq:x_k_in_colimbisbis} simplifies to
\[
x=\sum_{\substack{l<k<j\\ k\neq m}}y^N_{l,k}\oplus-F(l< k)(y^N_{l,k}).
\]
Repeating this process for every element $m\in \max K\setminus \supp(x)$ we find a decomposition similar to Equation \eqref{eq:x_k_in_colim},
\[
x=\sum_{l<k<j}y'_{l,k}\oplus-F(l< k)(y'_{l,k}),
\]
and satisfying that, for $K'=\{k\in \CP_{<j}\mid \exists\ l<k\mbox{ with }y'_{l,k}\not=0\}$, we have 
\[
\max K'\setminus \supp(x)< \max K\setminus \supp(x).
\]
Iterating this procedure we obtain a sequence of sets $\{K^n\}_{n\geq 0}$ and decompositions similar to Equation \eqref{eq:x_k_in_colim} with $K^0=K$, $K^1=K'$, and such that 
\[
\max K^{n+1}\setminus \supp(x)< \max K^n\setminus \supp(x).
\]
Setting $J^n=\max K^n\setminus \supp(x)$ and applying  {\cite[Lemma 2.6]{DiazRamos2009}} we find $N$ such that $J^N=\emptyset$, i.e., $\max K^N\subseteq \supp(x)$. For the corresponding decomposition,
\[
x=\sum_{l<k<j}\hat{y}_{l,k}\oplus-F(l< k)(\hat{y}_{l,k}),
\]
let $k$ belong to $\max \supp(x)$ so that we have, similarly to Equation \eqref{equ:long_computation3},
\[
x_k=\sum_{k<l}\hat{y}_{k,l}-\sum_{l<k}F(l<k)(\hat{y}_{l,k}).
\]
If $\sum_{k<l}\hat{y}_{k,l}\neq 0$, there exists some $l>k$ such that $\hat{y}_{l,k}\neq 0$, which is a contradiction with $\max K^N\subseteq \supp(x)$. Hence, $\sum_{k<l}\hat{y}_{k,l}=0$ and $x_k\in \im_F(k)$.
\end{proof}

\end{document}